\documentclass[a4paper]{article}

\usepackage{amsmath}
\usepackage{amssymb}
\usepackage{amsfonts}
\usepackage{bbm,amsmath,amsthm,tikz-cd}

\newcommand{\ZZ}{\mathbbm{Z}}

\newcommand{\PP}{\mathbbm{P}}

\newtheorem{theorem}{Theorem}
\newtheorem{lemma}[theorem]{Lemma}
\newtheorem{proposition}[theorem]{Proposition}

\begin{document}
\title{On syzygies of Segre embeddings of $\PP^1\times\PP^1$}
\author{Alexander Lemmens}
\date{\today}
\maketitle
\footnotetext[1]{The author is supported by the Flemish Research Council (FWO).}
\begin{abstract}
We construct a nonminimal graded free resolution of Segre embeddings of $\PP^1\times\PP^1$, although we don't compute all maps. We use this to prove an explicit formula for certain nonzero entries in the graded Betti table, at the end of the first row. We work over an arbitrary field $k$.
\end{abstract}
\tableofcontents
\section{Introduction}
We work over an arbitrary field $k$. Consider $\PP^1\times\PP^1$ as a variety embedded in $\PP^{(a+1)(b+1)-1}$ via the line bundle $O(a)\otimes O(b)$. This article is about the syzygies of this variety.\\

For any projective variety $X\subset\PP^{N-1}$ the homogeneous coordinate ring $R$ can be viewed as a module over the polynomial ring $S:=k[x_1,\ldots,x_N]$ and one can construct a minimal free graded resolution
$$0\longrightarrow F_d\longrightarrow\ldots\longrightarrow F_1\longrightarrow F_0\longrightarrow R,$$
at least in theory. In practise this can be very hard, and constructing a nonminimal resolution can be easier. But the minimal resolution is unique up to isomorphism.
So each $F_p$ is a graded free $S$-module and is a direct sum
$$F_p=\bigoplus_j S(-p-i_j),$$
and for each $p$ the number of $j$ such that $i_j=q$ is denoted $\kappa_{p,q}$.
By $S(-p-i_j)$ we mean $S$ considered as a graded module over itself where any generator has degree $p+i_j$. The numbers $\kappa_{p,q}$ are gathered in the graded Betti table:
\begin{center}
$$
\begin{array}{r|cccccc}
  & 0 & 1 & 2 & 3 & 4 & \ldots \\
\hline
0 & 1 & 0 & 0 & 0 & 0 & \dots  \\
1 & 0 & \kappa_{1,1} & \kappa_{2,1} & \kappa_{3,1} & \kappa_{4,1} & \dots  \\
2 & 0 & \kappa_{1,2} & \kappa_{2,2} & \kappa_{3,2} & \kappa_{4,2} & \dots  \\
3 & 0 & \kappa_{1,3} & \kappa_{2,3} & \kappa_{3,3} & \kappa_{4,3} & \dots  \\
\vdots & \vdots & \vdots & \vdots & \vdots & \vdots &  \\
\end{array}.$$
\end{center}
There are a few classes of algebraic varieties where the syzygies are completely known, for instance rational normal scrolls and varieties defined by certain monomial ideals \cite[p.\ 23]{eisenbud}. In most cases however this seems out of reach.
Elena Rubei studied property $N_p$ for Segre embeddings $\PP^{n_1}\times\ldots\times\PP^{n_d}$ in \cite{rubei,rubei2}. Thanks to the paper \cite{oeding} we know for any projective embedding of $\PP^1\times\ldots\times\PP^1$ exactly which entries of the Betti table are zero. We shall be concerned with $\PP^1\times\PP^1$.
Suppose
\begin{align}
\PP^1\times\PP^1&\longrightarrow\PP^{N-1}:\nonumber \\
((x:y),(z:w))&\longmapsto(x^az^b:x^{a-1}yz^b:\ldots:y^az^b:\ldots:x^aw^b:\ldots:y^aw^b)\nonumber
\end{align}
with $1\leq a\leq b$ and $N=(a+1)(b+1)$. Then the Betti table has the following shape:
\begin{center}
$$
\begin{array}{r|cccccc}
  & 0 & 1 & 2 & \ldots & N-3 & 0 \\
\hline
0 & 1 & 0 & 0 & \ldots & 0 & 0 \\
1 & 0 & \kappa_{1,1} & \kappa_{2,1} & \ldots & \kappa_{N-3,1} & 0 \\
2 & 0 & \kappa_{1,2} & \kappa_{2,2} & \ldots & \kappa_{N-3,2} & 0 \\
3 & 0 & 0 & 0 & 0 & 0 & 0
\end{array}$$
\end{center}
The entries $\kappa_{ab+b,1},\ldots,\kappa_{N-3,1}$ are all zero and so are the entries $\kappa_{1,2},\ldots,$ $\kappa_{2a+2b-3,2}$. All the remaining entries are nonzero.\\ 

We will construct a free graded resolution of this embedding when $3\leq a\leq b$. We won't explicitly construct all maps of this resolution but we will use it to obtain an explicit formula for $\kappa_{p,1}$ in the graded Betti table of this variety for $N-b-2\leq p$.
The strategy for constructing the free graded resolution is to embed our projective surface into a rational normal scroll of dimension $a+1$, which is scrolled out by $a+1$ rational normal curves of degree $b$. We then construct a resolution of our surface over this scroll and then take a free resolution of each module in the resolution to get a double complex. We then perform an iterated mapping cone construction to obtain a free resolution of our surface. This was inspired by what Schreyer did in \cite{schreyer} to obtain resolutions of projective curves.
\begin{theorem}\label{firstrow}
Let $3\leq a\leq b$ then we have the following
\begin{itemize}
\item[$i$] For all $ab+a\leq p$ we have $\kappa_{p,1}=p\binom{(a+1)b}{p+1}$.
\item[$ii$] For $p=ab+a-1$ we have $\kappa_{p,1}=p\binom{(a+1)b}{p+1}+p$.
\end{itemize}
\end{theorem}
The source of the syzygies in the above theorem is the embeddings
$$\PP^1\times\PP^1\longrightarrow\PP^1\times\PP^b\longrightarrow\PP^{N-1},$$
$$\PP^1\times\PP^1\longrightarrow\PP^a\times\PP^1\longrightarrow\PP^{N-1},$$
where $N=(a+1)(b+1))$. Note that $\PP^a\times\PP^1$ and $\PP^1\times\PP^b$ are rational normal scrolls in $\PP^{N-1}$ and their Betti tables are known \cite[pp.\ 110-112]{schreyer}. For every $p\geq 0$ this induces inequalities
$$p\binom{(a+1)b}{p+1}=\kappa_{p,1}(\PP^a\times\PP^1)\leq\kappa_{p,1}(\PP^1\times\PP^1),$$
$$p\binom{(b+1)a}{p+1}=\kappa_{p,1}(\PP^1\times\PP^b)\leq\kappa_{p,1}(\PP^1\times\PP^1),$$
and this is exactly where these syzygies are coming from.\\

One can view $\PP^1\times\PP^1$ as a toric variety and take the bidegree decomposition of the syzygy spaces with respect to the torus action (see \cite[p.\ 8]{computingbetti}).
In the case $a=b\geq 3$ we find that $\kappa_{a^2+a-1,1}=2(a^2+a-1)$ and the bidegree table of this entry takes the following form:
\begin{center}
\begin{tabular}{ccccccccccc}
0&0&0&0&0&1&0&0&0&0&0\\
0&0&0&0&0&1&0&0&0&0&0\\
0&0&0&0&0&1&0&0&0&0&0\\
0&0&0&0&0&1&0&0&0&0&0\\
0&0&0&0&0&1&0&0&0&0&0\\
1&1&1&1&1&2&1&1&1&1&1\\
0&0&0&0&0&1&0&0&0&0&0\\
0&0&0&0&0&1&0&0&0&0&0\\
0&0&0&0&0&1&0&0&0&0&0\\
0&0&0&0&0&1&0&0&0&0&0\\
0&0&0&0&0&1&0&0&0&0&0\\
\end{tabular}
\end{center}
This is the case $a=b=3$, but the cross pattern appears for all $a=b\geq 3$ (see proposition \ref{cross}).\newline
For $a=b=2$ one has $\kappa_{a^2+a-1,1}=20$ and bidegree table is:
\begin{center}
\begin{tabular}{ccccc}
0&0&1&0&0\\
0&1&2&1&0\\
1&2&4&2&1\\
0&1&2&1&0\\
0&0&1&0&0\\
\end{tabular}
\end{center}
In the case $2=a\leq b$ the entire Betti-table is known. We have
$$\kappa_{p,2}=\max(p-2b-1,0)\binom{3b}{p},\hspace{5 pt}\kappa_{p,1}=\kappa_{p-1,2}+p\binom{3b+2}{p+1}-4b\binom{3b}{p-1}.$$
This follows from \cite[Corollary 4, p.\ 3]{lemmens}.\\

In this article all modules will be graded and all morphisms degree-preserving.
\subsection*{Acknowledgements}
This research is part of my PhD which is funded by the Flemish Research Council (FWO).
\section{The Eagon-Northcott complex}
For a rational normal scroll $X$ with invariants $f=e_1+\ldots+e_\ell$, the homogeneous coordinate ring is
$$R_{e_1,\ldots,e_\ell}:=\bigoplus_{n\geq 0}\Big(\bigoplus_{x_{i_1}\ldots x_{i_n}\in k[x_1,\ldots,x_\ell]}V_{i_1,\ldots,i_n}\Big)$$
where $V_{i_1,\ldots,i_n}$ is a vector space of dimension $e_{i_1}+\ldots+e_{i_n}+1$ whose basis elements we denote by $b_{i_1,\ldots,i_n,0},\ldots,b_{i_1,\ldots,i_n,e_{i_1}+\ldots+e_{i_n}}$. Note that $b_{i_1,\ldots,i_n,j}$ does not depend on the order of $i_1,\ldots,i_n$. Multiplication in this ring is defined by the rule
$$b_{i_1,\ldots,i_n,j}\cdot b_{i'_1,\ldots,i'_m,j'}=b_{i_1,\ldots,i_n,i'_1,\ldots,i'_m,j+j'}.$$
The second direct sum is over the monomials of degree $n$ in the variables $x_1,\ldots,x_\ell$\\
%one has the free module $F$ of rank $f-\ell$ with a morphism to the rank 2 free module $G$, given by:

\noindent\textbf{Example}\vspace{3 pt}\newline
Suppose $\ell=3$ $e_1=2$, $e_2=3$, $e_3=3$ and $f=8$ then we have the variables $x_1$, $x_2$ and $x_3$. The following picture represents the homogeneous coordinate ring up to the degree two part, where each dot corresponds to a basis element.
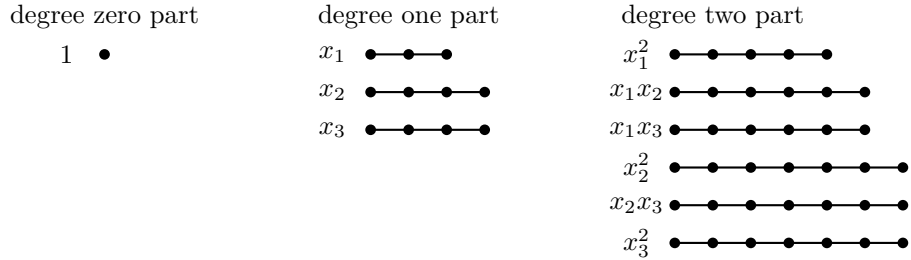
\begin{center}
\begin{tikzpicture}
\draw[thick] (4,2.5) -- (5,2.5);
\draw[thick] (4,2) -- (5.5,2);
\draw[thick] (4,1.5) -- (5.5,1.5);
\draw[thick] (8,2.5) -- (10,2.5);
\draw[thick] (8,2) -- (10.5,2);
\draw[thick] (8,1.5) -- (10.5,1.5);
\draw[thick] (8,1) -- (11,1);
\draw[thick] (8,.5) -- (11,.5);
\draw[thick] (8,0) -- (11,0);
\fill[black] (0.5,2.5) circle (2 pt);
\fill[black] (4,2.5) circle (2 pt);
\fill[black] (4.5,2.5) circle (2 pt);
\fill[black] (5,2.5) circle (2 pt);
\fill[black] (4,2) circle (2 pt);
\fill[black] (4.5,2) circle (2 pt);
\fill[black] (5,2) circle (2 pt);
\fill[black] (5.5,2) circle (2 pt);
\fill[black] (4,1.5) circle (2 pt);
\fill[black] (4.5,1.5) circle (2 pt);
\fill[black] (5,1.5) circle (2 pt);
\fill[black] (5.5,1.5) circle (2 pt);
\fill[black] (8,2.5) circle (2 pt);
\fill[black] (8.5,2.5) circle (2 pt);
\fill[black] (9,2.5) circle (2 pt);
\fill[black] (9.5,2.5) circle (2 pt);
\fill[black] (10,2.5) circle (2 pt);
\fill[black] (8,2) circle (2 pt);
\fill[black] (8.5,2) circle (2 pt);
\fill[black] (9,2) circle (2 pt);
\fill[black] (9.5,2) circle (2 pt);
\fill[black] (10,2) circle (2 pt);
\fill[black] (10.5,2) circle (2 pt);
\fill[black] (8,1.5) circle (2 pt);
\fill[black] (8.5,1.5) circle (2 pt);
\fill[black] (9,1.5) circle (2 pt);
\fill[black] (9.5,1.5) circle (2 pt);
\fill[black] (10,1.5) circle (2 pt);
\fill[black] (10.5,1.5) circle (2 pt);
\fill[black] (8,1) circle (2 pt);
\fill[black] (8.5,1) circle (2 pt);
\fill[black] (9,1) circle (2 pt);
\fill[black] (9.5,1) circle (2 pt);
\fill[black] (10,1) circle (2 pt);
\fill[black] (10.5,1) circle (2 pt);
\fill[black] (11,1) circle (2 pt);
\fill[black] (8,.5) circle (2 pt);
\fill[black] (8.5,.5) circle (2 pt);
\fill[black] (9,.5) circle (2 pt);
\fill[black] (9.5,.5) circle (2 pt);
\fill[black] (10,.5) circle (2 pt);
\fill[black] (10.5,.5) circle (2 pt);
\fill[black] (11,.5) circle (2 pt);
\fill[black] (8,0) circle (2 pt);
\fill[black] (8.5,0) circle (2 pt);
\fill[black] (9,0) circle (2 pt);
\fill[black] (9.5,0) circle (2 pt);
\fill[black] (10,0) circle (2 pt);
\fill[black] (10.5,0) circle (2 pt);
\fill[black] (11,0) circle (2 pt);
\node at (0,2.5) {1};
\node at (3.5,2.5) {$x_1$};
\node at (3.5,2) {$x_2$};
\node at (3.5,1.5) {$x_3$};
\node at (7.5,2.5) {$x_1^2$};
\node at (7.5,2) {$x_1x_2$};
\node at (7.5,1.5) {$x_1x_3$};
\node at (7.5,1) {$x_2^2$};
\node at (7.5,.5) {$x_2x_3$};
\node at (7.5,0) {$x_3^2$};
\node at (.5,3) {degree zero part};
\node at (4.5,3) {degree one part};
\node at (8.5,3) {degree two part};
\end{tikzpicture}
\end{center}
In our notation the second dot in the first row of the degree one part is denoted $b_{1,1}$ and the third dot in that same row is denoted $b_{1,2}$. Multiplying them in this ring yields $b_{1,1,3}$ which is the fourth dot in the first row in the degree two part.\\

For any $c\geq 0$ one has the $R_{e_1,\ldots,e_\ell}$-module
$$M_{e_1,\ldots,e_\ell,c}:=\bigoplus_{n\geq 0}\Big(\bigoplus_{x_{i_1}\ldots x_{i_n}\in k[x_1,\ldots,x_\ell]}V_{i_1,\ldots,i_n,c}\Big)$$
where $V_{i_1,\ldots,i_n,c}$ is a vector space whose basis elements we denote by $b_{i_1,\ldots,i_n,c,j}$ with $0\leq j\leq e_{i_1}+\ldots+e_{i_n}+c$. Again it does not depend on the order of the $i_1,\ldots,i_n$. Scalar multiplication is defined via the rule
$$b_{i_1,\ldots,i_n,j}\cdot b_{i'_1,\ldots,i'_n,c,j'}=b_{i_1,\ldots,i_n,i'_1,\ldots,i'_n,c,j+j'}.$$
Obviously the case $c=0$ just gives $R_{e_1,\ldots,e_\ell}$.\newline
The Eagon-Northcott complex gives a minimal free graded resolution of these modules. Note that these free modules are over the polynomial ring in $(e_1+1)+\ldots+(e_\ell+1)$ variables, which we identify with the $b_{j,i}\in R_{e_1,\ldots,e_\ell}$ of degree one. Note that $R_{e_1,\ldots,e_\ell}$ and $M_{e_1,\ldots,e_\ell,c}$ can also be viewed as graded modules over this polynomial ring.\newline
To construct this resolution we first introduce the free module $F$ of rank $e_1+\ldots+e_\ell$ over our polynomial ring. We denote its basis elements $f_{i,j}$, where $1\leq i\leq\ell$ and $1\leq j\leq e_i$. We can now define the minimal free graded resolution of $M_c:=M_{e_1,\ldots,e_\ell,c}$:
\begin{align}
\bigwedge^{f}F\otimes V_{f-c-2}(-f)&\longrightarrow\ldots\longrightarrow\bigwedge^{c+2}F(-c-2)\longrightarrow\bigwedge^cF(-c)\nonumber \\
\longrightarrow\bigwedge^{c-1}F\otimes V_{1}(-c+1)&\longrightarrow\ldots\longrightarrow\bigwedge^1F\otimes V_{c-1}(-1)\longrightarrow V_c\longrightarrow M_c.\nonumber
\end{align}
Here $V_a$ means a rank $a+1$ free graded module whose basis elements we label $B_0,\ldots,B_a$. The paranthesized integers mean that the degrees are shifted to make the morphisms degree-preserving. The morphisms in the first row are defined by
\begin{align}
\bigwedge^mF\otimes V_n(-m)\longrightarrow&\bigwedge^{m-1}F\otimes V_{n-1}(-m+1):\hspace{10 pt}f_{i_1,j_1}\wedge\ldots\wedge f_{i_m,j_m}\otimes B_p\longmapsto\nonumber \\
&\sum_{a=1}^m(-1)^ab_{i_a,j_a-1}\cdot f_{i_1,j_1}\wedge\ldots\wedge\widehat{f_{i_a,j_a}}\wedge\ldots\wedge f_{i_m,j_m}\otimes B_{p-1}\nonumber \\
-&\sum_{a=1}^m(-1)^ab_{i_a,j_a}\cdot f_{i_1,j_1}\wedge\ldots\wedge\widehat{f_{i_a,j_a}}\wedge\ldots\wedge f_{i_m,j_m}\otimes B_p\nonumber
\end{align}
where the hat means that term is deleted. The dot means scalar multiplication of the module and $b_{i_a,j_a}$ means the variable in the polynomial ring over which we work corresponding to the degree one element $b_{i_a,j_a}$ of $R_{e_1,\ldots,e_\ell}$. If the index of $B_p$ or $B_{p-1}$ is out of range then that term is deleted. This happens for $B_n$ and $B_{-1}$.\newline
The last map on the first row is given by
\begin{align}
\bigwedge^{c+2}F(-c-2)&\longrightarrow\bigwedge^cF(-c):\hspace{10 pt}f_{i_1,j_1}\wedge\ldots\wedge f_{i_{c+2},j_{c+2}}\longmapsto\nonumber \\
&\sum_{1\leq a_1<a_2\leq c+2}(-1)^{a_1+a_2}(b_{i_{a_1},j_{a_1}}b_{i_{a_2},j_{a_2}-1}-b_{i_{a_1},j_{a_1}-1}b_{i_{a_2},j_{a_2}})
\nonumber \\
\cdot &f_{i_1,j_1}\wedge\ldots\wedge\widehat{f_{i_{a_1},j_{a_1}}}\wedge\ldots\wedge\widehat{f_{i_{a_2},j_{a_2}}}\wedge\ldots\wedge f_{i_{c+2},j_{c+2}}.\nonumber
\end{align}
Again the expression with the $b$'s lives in the polynomial ring. Finally the maps on the second row are given by
\begin{align}
\bigwedge^mF\otimes V_n(-m)\longrightarrow&\bigwedge^{m-1}F\otimes V_{n+1}(-m+1):\hspace{10 pt}f_{i_1,j_1}\wedge\ldots\wedge f_{i_m,j_m}\otimes B_p\longmapsto\nonumber \\
&\sum_{a=0}^m(-1)^ab_{i_a,j_a-1}\cdot f_{i_1,j_1}\wedge\ldots\wedge\widehat{f_{i_a,j_a}}\wedge\ldots\wedge f_{i_m,j_m}\otimes B_{p+1}\nonumber \\
-&\sum_{a=0}^m(-1)^ab_{i_a,j_a}\cdot f_{i_1,j_1}\wedge\ldots\wedge\widehat{f_{i_a,j_a}}\wedge\ldots\wedge f_{i_m,j_m}\otimes B_p\nonumber
\end{align}
To define the map from $V_c$ to $M_c=M_{e_1,\ldots,e_\ell,c}$ note that the degree zero part of $M_c$ is just $V_c$ (as a vector space), so we map $B_p\in V_c$ to $b_{c,p}\in M_c$.\newline
See \cite[pp.\ 110-112]{schreyer} for more information on the Eagon-Northcott complex.
\section{The relative resolution}
We are interested in the projective variety $\PP^1\times\PP^1\rightarrow\PP^{N-1}$ introduced in the beginning of the article, depending on parameters $1\leq a\leq b$. Its homogeneous coordinate ring is
$$R_{a,b}=\bigoplus_{n\geq 0}V_{\{0,\ldots,na\}\times\{0,\ldots,nb\}},$$
Where  $V_{\{0,\ldots,na\}\times\{0,\ldots,nb\}}$ is a vector space whose basis elements we denote $b_{n,i,j}$ with $0\leq i\leq na$ and $0\leq i\leq nb$.\newline
This variety is contained in the rational normal scroll with $\ell=a+1$, $e_1=\ldots=e_\ell=b$. The corresponding surjective morphism of homogeneous coordinate rings is
\begin{align}
R_{\bar{b}}=\bigoplus_{n\geq 0}\Big(\bigoplus_{x_{i_1}\ldots x_{i_n}\in k[x_0,\ldots,x_a]}V_{i_1,\ldots,i_n}\Big)&\longrightarrow\bigoplus_{n\geq 0}V_{\{0,\ldots,na\}\times\{0,\ldots,nb\}}
\nonumber \\
b_{i_1,\ldots,i_n,j}&\longmapsto(i_1+\ldots+i_n,j)\label{scrollmap}
\end{align}
where $R_{\bar{b}}$ is $R_{b,\ldots,b}$ with $a+1$ copies of $b$. As you can see we have taken the monomials from $x_0$ to $x_a$ for convenience (rather than from $x_1$ to $x_{a+1}$). We now introduce the resolution of $R_{a,b}$ by $R_{\bar{b}}$-modules:
\begin{align}
M_{\bar{b},ab}\otimes\bigwedge^aG_{a}\otimes G_{a-1}(-a)&\longrightarrow\ldots\longrightarrow M_{\bar{b},3b}\otimes\bigwedge^3G_{a}\otimes G_2(-3)\longrightarrow\nonumber \\
M_{\bar{b},2b}\otimes\bigwedge^2G_{a}\otimes G_1(-2)&\longrightarrow R_{\bar{b}}\longrightarrow R_{a,b}\label{relres}
\end{align}
where $G_i$ is just a vector space with basis $g_1,\ldots,g_i$. Note that $\otimes G_1$ is redundant as $G_1$ is one-dimensional. The maps on the first row are given by
\begin{align}
M_{\bar{b},pb}\otimes\bigwedge^pG_a&\otimes G_{p-1}(-p)\longrightarrow M_{\bar{b},(p-1)b}\otimes\bigwedge^{p-1}G_a\otimes G_{p-2}(-p+1)\nonumber \\
&b_{i_1,\ldots,i_n,pb,j}\otimes g_{j_1}\wedge\ldots\wedge g_{j_p}\otimes g_h\longmapsto\nonumber \\
&\sum_{q=1}^p(-1)^qb_{i_1,\ldots,i_n,j_q-1,p(b-1),j}\otimes g_{j_1}\wedge\ldots\wedge\widehat{g_{j_q}}\wedge\ldots\wedge g_{j_p}\otimes g_{h-1}\nonumber \\
-&\sum_{q=1}^p(-1)^qg_{i_1,\ldots,i_n,j_q,p(b-1),j}\otimes g_{j_1}\wedge\ldots\wedge\widehat{g_{j_q}}\wedge\ldots\wedge g_{j_p}\otimes g_h.\nonumber
\end{align}
If the index of $g_h$ or $g_{h-1}$ is out of range then that term is deleted. This happens for $g_0$ and $g_p$.
The map from $M_{\bar{b},2b}\otimes\bigwedge^2G_{a}(-2)$ to $R_{\bar{b}}$ is given by
$$b_{i_1,\ldots,i_n,2b,j}\otimes g_{j_1}\wedge g_{j_2}\longmapsto b_{i_1,\ldots i_n,j_1-1,j_2,j}-b_{i_1,\ldots i_n,j_1,j_2-1,j}.$$
\begin{lemma}
The relative resolution (\ref{relres}) is exact.
\end{lemma}
\begin{proof}
The degree $n$ part of the resolution is
\begin{align}
&\Big(\bigoplus_{x_{i_1}\ldots x_{i_{n-a}}\in k[x_0,\ldots,x_a]}V_{i_1,\ldots,i_{n-a},ab}\Big)\otimes\bigwedge^aG_a\otimes G_{a-1}\longrightarrow\ldots\nonumber \\
\longrightarrow&\Big(\bigoplus_{x_{i_1}\ldots x_{i_{n-2}}\in k[x_0,\ldots,x_a]}V_{i_1,\ldots,i_{n-2},2b}\Big)\otimes\bigwedge^2G_{a}\otimes G_1\nonumber \\
\longrightarrow&\Big(\bigoplus_{x_{i_1}\ldots x_{i_n}\in k[x_0,\ldots,x_a]}V_{i_1,\ldots,i_n}\Big)
\longrightarrow V_{\{0,\ldots,na\}\times\{0,\ldots,nb\}}. \nonumber
\end{align}
We have to prove that this is exact. Now note that $V_{i_1,\ldots,i_{n-p},pb}$ is always isomorphic to $V_{nb}$; the isomorphism is given by $b_{i_1,\ldots,i_{n-p},pb,j}\mapsto b_{i_1,\ldots,i_n,j}$. In fact we can write this whole sequence as the tensor product of $V_{nb}$ with the following sequence:
$$
S^{n-a}(V_a)\otimes\bigwedge^aG_a\otimes G_{a-1}\longrightarrow\ldots\longrightarrow S^{n-2}(V_a)\otimes\bigwedge^2G_a
\longrightarrow S^n(V_a)\longrightarrow V_{na}.\nonumber
$$
Note that $S^nV_a$ is the degree $n$ part of $k[x_0,\ldots,x_a]$. So we have to prove that this sequence is exact. We reduce it to the exactness of the Eagon-Northcott complex by observing that this sequence is isomorphic to the degree $n$ part of the Eagon-Northcott complex of $R_a$ with $\ell=1$ and $e_1=a$. Note that $\bigwedge^{i+1}G_a\otimes G_i$ will correspond to the degree zero part of $\bigwedge^{i+1}F\otimes V_i$ and that the degree $n$ part consists of this tensored with the degree $n$ part of the polynomial ring in the variables $b_1,\ldots,b_{a+1}$, which is isomorphic to $S^n(V_a)$.
\end{proof}
\section{The iterated mapping cone}
In general whenever $R$ is a qoutient ring of a polynomial ring $S$ and $M$ is an $R$-module with a resolution
$P_n\rightarrow P_{n-1}\rightarrow\ldots\rightarrow P_0\rightarrow M$ of $R$-modules each of which has a free $S$-resolution $Q_{i,m}\rightarrow Q_{i,m-1}\rightarrow\ldots\rightarrow Q_{i,0}\rightarrow P_i$, then one can build a chain map
\begin{equation}\label{first_chain_map}
\begin{tikzcd}
P_n \arrow{r} & P_{n-1} \arrow{r} & \ldots \arrow{r} & P_0 \arrow{r} & M\\
Q_{n,0} \arrow{r}\arrow{u} & Q_{n-1,0} \arrow{u} & \ldots & Q_{0,0}\arrow{u} & \\
\vdots\arrow{u} & \vdots\arrow{u} & & \vdots\arrow{u} & \\
Q_{n,m-1} \arrow{r}\arrow{u} & Q_{n-1,m-1} \arrow{u} & \ldots & Q_{0,m-1}\arrow{u} & \\
Q_{n,m} \arrow{r}\arrow{u} & Q_{n-1,m} \arrow{u} & \ldots & Q_{0,m}\arrow{u} & 
\end{tikzcd}
\end{equation}
and take its mapping cone. If $I_{n-1}$ is the image of the morphism $P_{n-1}\rightarrow P_{n-2}$ then the following mapping cone gives a resolution of $I_{n-1}$:
\begin{center}
\begin{tikzcd}
P_n \arrow{r} & P_{n-1} \arrow{r} & I_{n-1}\\
Q_{n,0} \arrow{r}\arrow{u} & Q_{n-1,0} \arrow{r}\arrow{u} & Q_{n-1,0}\arrow{u} \\
Q_{n,1} \arrow{r}\arrow{u} & Q_{n-1,1} \arrow{r}\arrow{u} & Q_{n,0}\oplus Q_{n-1,1}\arrow{u} \\
\vdots\arrow{u} & \vdots\arrow{u} & \vdots\arrow{u} \\
Q_{n,m} \arrow{r}\arrow{u} & Q_{n-1,m} \arrow{r}\arrow{u} & Q_{n,m-1}\oplus Q_{n-1,m}\arrow{u} \\
 &  & Q_{n,m}\arrow{u} \\
\end{tikzcd}
\end{center}
One can then build a chain map from this resolution of $I_{n-1}$ to the resolution of $P_{n-2}$.
Again taking the mapping cone one gets a resolution of the image $I_{n-2}$ of $P_{n-1}\rightarrow P_{n-2}$:
\begin{center}
\begin{tikzcd}
I_{n-1} \arrow{r} & P_{n-2} \arrow{r} & I_{n-2}\\
Q_{n-1,0} \arrow{r}\arrow{u} & Q_{n-2,0} \arrow{r}\arrow{u} & Q_{n-2,0}\arrow{u} \\
Q_{n,0}\oplus Q_{n-1,1} \arrow{r}\arrow{u} & Q_{n-2,1} \arrow{r}\arrow{u} & Q_{n-1,0}\oplus Q_{n-2,1}\arrow{u} \\
Q_{n,1}\oplus Q_{n-1,2} \arrow{r}\arrow{u} & Q_{n-2,2} \arrow{r}\arrow{u} & Q_{n,0}\oplus Q_{n-1,1}\oplus Q_{n-2,2}\arrow{u} \\
\vdots\arrow{u} & \vdots\arrow{u} & \vdots\arrow{u}
\end{tikzcd}
\end{center}
If one keeps repeating this process one eventually gets a resolution of $M$. This resolution is not necessarily minimal, even if the resolutions of the $P_i$ are all minimal. To get the graded Betti table of $M$ all one has to do is tensor this resolution with the field $k$ over which we work and then take homology of the resulting complex.
The $p$-th module in this resolution is
$$C_p:=\bigoplus_{\max(0,p-n)\leq j\leq\min(p,m)} Q_{p-j,j}.$$
The map from $C_p$ to $C_{p-1}$ is given by a matrix of maps from $Q_{i,j}$ to $Q_{i',j'}$ where $i+j=p$ and $i'+j'=p-1$.
\begin{lemma}
Up to sign the map from $Q_{n,j}$ to $Q_{n-1,j}$ is given by our initial chain map from the resolution of $P_n$ to our resolution of $P_{n-1}$. 
\end{lemma}
\begin{proof}
This follows from the definition of the mapping cone. For instance the resolution of $I_{n-1}$ above has maps
$$Q_{n,j}\oplus Q_{n-1,j+1}\longrightarrow Q_{n,j-1}\oplus Q_{n-1,j},$$
where the map from $Q_{n-1,j+1}$ to $Q_{n,j-1}$ is zero, and the other maps are given by the diagram (\ref{first_chain_map}), except that the sign of the map $Q_{n,j}\rightarrow Q_{n,j-1}$ is flipped.
The maps in the chain map from the resolution of $I_{n-1}$ to the one of $P_{n-2}$ all go from a $Q_{n,j}\oplus Q_{n-1,j+1}$ to a $Q_{n-2,j+1}$. So the maps in the corresponding mapping cone consist only of the map from $Q_{n-2,j+2}$ to $Q_{n-2,j+1}$, the maps from $Q_{n,j}$ and $Q_{n-1,j+1}$ to $Q_{n-2,j+1}$ and the maps we already got from the previous mapping cone, but with the sign flipped. At the end of this process in the resolution of $M$ we still have the maps from $Q_{n,j}$ to $Q_{n-1,j}$, with sign $(-1)^n$.
\end{proof}
%Note that we are in a situation where we can apply this lemma because the Eagon-Northcott complex is a minimal resolution.
\section{The horizontal chain maps}
When we take the relative resolution
$$M_{\bar{b},ab}\otimes\bigwedge^aG_{a}\otimes G_{a-1}(-a)\longrightarrow
M_{\bar{b},(a-1)b}\otimes\bigwedge^{a-1}G_{a}\otimes G_{a-2}(-a+1)\longrightarrow\ldots$$
and we take a free resolution of each module in this resolution, we get something like this:
\begin{center}
\begin{tikzcd}
\underset{\otimes\bigwedge^aG_{a}\otimes G_{a-1}(-a)}{M_{\bar{b},ab}} \arrow{r} &
\underset{\otimes\bigwedge^{a-1}G_{a}\otimes G_{a-2}(-a+1)}{M_{\bar{b},(a-1)b}} \arrow{r} & \ldots \\
\underset{\otimes\bigwedge^aG_{a}\otimes G_{a-1}(-a)}{V_{ab}} \arrow[dotted]{r}\arrow{u} & 
\underset{\otimes\bigwedge^{a-1}G_{a}\otimes G_{a-2}(-a+1)}{V_{(a-1)b}} \arrow[dotted]{r}\arrow{u} & \ldots \\
\vdots\arrow{u} & \vdots\arrow{u} & \\
\underset{\otimes \bigwedge^aG_{a}\otimes G_{a-1}(-ab-a-b+1)}{\bigwedge^{(a+1)b-1}F\otimes V_{b-3}} \arrow[dotted]{r}\arrow{u} &
\underset{\otimes \bigwedge^{a-1}G_{a}\otimes G_{a-2}(-ab-a-b+2)}{\bigwedge^{(a+1)b-1}F\otimes V_{2b-3}} \arrow[dotted]{r}\arrow{u} & \ldots \\
\underset{\otimes \bigwedge^aG_{a}\otimes G_{a-1}(-ab-a-b)}{\bigwedge^{(a+1)b}F\otimes V_{b-2}}\arrow[dotted]{r}\arrow{u} & 
\underset{\otimes \bigwedge^{a-1}G_{a}\otimes G_{a-2}(-ab-a-b+1)}{\bigwedge^{(a+1)b}F\otimes V_{2b-2}}\arrow[dotted]{r}\arrow{u} & \ldots \\
\end{tikzcd}
\end{center}
Now we have to construct the dotted maps so that the diagram commutes. We will not compute all maps of the iterated mapping cone as this is complicated and not necessary for our purposes. We just compute a chain map from each resolution of $M_{\bar{b},pb}\otimes\bigwedge^pG_{a}\otimes G_{p-1}(-p)$ to the next resolution. This does not give a double complex because the composition of two consecutive horizontal maps is not zero. Put $c=pb$. Each module $M_{\bar{b},c}\otimes\bigwedge^pG_{a}\otimes G_{p-1}(-p)$ is a direct sum of copies of $M_{\bar{b},c}(-p)$ indexed by expressions of the form
$g_{i_1}\wedge\ldots\wedge g_{i_p}\otimes g_j$
which form a basis of $\bigwedge^pG_{a}\otimes G_{p-1}$.
This means that every map $M_{\bar{b},c}\otimes\bigwedge^pG_{a}\otimes G_{p-1}(-p)\rightarrow M_{\bar{b},c-b}\otimes\bigwedge^{p-1}G_{a}\otimes G_{p-2}(-p+1)$ can be viewed as a matrix whose entries are maps from $M_{\bar{b},c}(-p)$ to $M_{\bar{b},c-b}(-p+1)$. It is therefore enough to construct a chain map corresponding to each (non-zero) entry of this matrix. Each such map from $M_{\bar{b},c}(-p)$ to $M_{\bar{b},c-b}(-p+1)$ is given by
$$\beta_{j_0}(b_{i_1,\ldots,i_n,c,j})=b_{i_1,\ldots,i_n,j_0,c-b,j}$$
for some $j_0\in\{0,\ldots,a\}$. So we have to construct the dotted maps in the following diagram:
\begin{center}
\begin{tikzcd}
M_{\bar{b},c} \arrow{r}{\beta_{i_0}} & M_{\bar{b},c-b} \\
V_{c} \arrow[dotted]{r}{\alpha_{i_0,0}}\arrow{u} & V_{c-b}\arrow{u} \\
\bigwedge^1F\otimes V_{c-1} \arrow[dotted]{r}{\alpha_{i_0,1}}\arrow{u} &
\bigwedge^1F\otimes V_{c-b-1} \arrow{u}\\
\bigwedge^2F\otimes V_{c-2}\arrow[dotted]{r}{\alpha_{i_0,2}}\arrow{u} & 
\bigwedge^2F\otimes V_{c-b-2} \arrow{u}\\
\vdots\arrow{u} & \vdots\arrow{u} & \\
\end{tikzcd}
\end{center}
where we have suppressed the degree shifts from our notation.
Note that it is enough to define these maps on the basis elements, because these are free modules. For $\alpha_{j_0,0}$ we set
$$\alpha_{i_0,0}(B_j)=\begin{cases}
b_{i_0,j}B_0& \text{if }j\leq b \\
b_{i_0,b}B_{j-b}& \text{if }j\geq b
\end{cases}$$
Recall that we are working over the polynomial ring whose variables are the degree one elements $b_{i,j}$ of $R_{\bar{b}}$. This expression is scalar multiplication of such a variable with a basis element of the free module $V_{c-b}$ with basis $B_0,\ldots,B_{c-b}$. The above square commutes because $b_{i_0,j}B_0$ gets mapped to $b_{i_0,j}b_{c-b,0}=b_{i_0,c-b,j}=\beta_{i_0}(b_{c,j})$ and $b_{i_0,b}B_{j-b}$ gets mapped to $b_{i_0,b}b_{c-b,j-b}=b_{i_0,c-b,j}=\beta_{i_0}(b_{c,j})$, both in $M_{\bar{b},c-b}$.\newline
Now for any $0\leq n\leq c-b$ we define $\alpha_{i_0,n}$ as follows:
$$\alpha_{i_0,n}(f_{i_1,j_1}\wedge\ldots\wedge f_{i_n,j_n}\otimes B_j)=$$
$$\begin{cases}
b_{i_0,j} f_{i_1,j_1}\wedge\ldots\wedge f_{i_n,j_n}\otimes B_0+ &\\
\sum_{\ell=1}^n(-1)^\ell b_{i_\ell,j_\ell-1}f_{i_0,j+1}\wedge f_{i_1,j_1}\wedge\ldots\wedge\widehat{f_{i_\ell,j_\ell}}\wedge\ldots\wedge f_{i_n,j_n}\otimes B_0& \text{if } j<b \\
b_{i_0,b}f_{i_1,j_1}\wedge\ldots\wedge f_{i_n,j_n}\otimes B_{j-b}& \text{if }j\geq b
\end{cases}$$
And for $c-b<n\leq c$ we define
$\alpha_{i_0,n}:\bigwedge^nF\otimes V_{c-n}\rightarrow\bigwedge^{n+1}F\otimes V_{n-c+b-1}$ as follows:
$$
\alpha_{i_0,n}(f_{i_1,j_1}\wedge\ldots\wedge f_{i_n,j_n}\otimes B_j)=
$$
\begin{equation}\label{horizontalmap}
(-1)^{n-c+b-1}\sum_{\ell=0}^{n-c+b-1}f_{i_0,j+\ell+1}\wedge f_{i_1,j_1}\wedge\ldots\wedge f_{i_n,j_n}\otimes B_{\ell}
\end{equation}
For $n>c$ we don't explicitly construct the horizontal maps as we don't need them. (Of course they exist.)
Verifying that the squares commute is left as an exercise to the reader, except for the following square which we will treat as an example:
\begin{center}
\begin{tikzcd}
\bigwedge^{c-b}F\otimes V_b \arrow{r}{\alpha_{i_0,c-b}} &
\bigwedge^{c-b}F\otimes V_0 \\
\bigwedge^{c-b+1}F\otimes V_{b-1}\arrow{r}{\alpha_{i_0,c-b+1}}\arrow{u} & 
\bigwedge^{c-b+2}F\otimes V_0 \arrow{u}
\end{tikzcd}
\end{center}
We begin in the lower left corner. Let $f_{i_1,j_1}\wedge\ldots\wedge f_{i_{c-b+1},j_{c-b+1}}\otimes B_j\in \bigwedge^{c-b+1}F\otimes V_{b-1}$. We will abbreviate $f_{i_1,j_1}\wedge\ldots\wedge\widehat{f_{i_\ell,j_\ell}}\wedge\ldots\wedge f_{i_{c-b+1},j_{c-b+1}}$ by $\ldots\wedge\widehat{f_{i_\ell,j_\ell}}\wedge\ldots$.
When we go to the upper left we get:
$$\sum_{\ell=1}^{c-b+1}(-1)^\ell(b_{i_\ell,j_\ell-1}\ldots\wedge\widehat{f_{i_\ell,j_\ell}}\wedge\ldots\otimes B_{k+1}
-b_{i_\ell,j_\ell}\ldots\wedge\widehat{f_{i_\ell,j_\ell}}\wedge\ldots\otimes B_k)$$
When we go to the upper right, we have to treat two cases: $j<b-1$ and $j=b-1$. Before we do so we introduce the following notation: for any two integers $m,l$ we set
$$\epsilon_{m,l}=\begin{cases}
1 &\text{if }m<l \\
0 &\text{if }m=l \\
-1&\text{if }m>l.
\end{cases}$$
We now treat the first case. When we go to the upper right we get:
\begin{align}
&\sum_{\ell=1}^{c-b+1}(-1)^\ell( b_{i_0,j+1}b_{i_\ell,j_\ell-1}\ldots\wedge\widehat{f_{i_\ell,j_\ell}}\wedge\ldots\otimes B_0\nonumber \\
&\hspace{50 pt}-b_{i_0,j}b_{i_\ell,j_\ell}\ldots\wedge\widehat{f_{i_\ell,j_\ell}}\wedge\ldots\otimes B_0)\nonumber \\
&+\sum_{\ell,m=1}^{c-b+1}(-1)^{\ell+m}\epsilon_{m,l}b_{i_\ell,j_\ell-1}b_{i_m,j_m-1}f_{i_0,j+2}\wedge
\ldots\wedge\widehat{f_{i_\ell,j_\ell}}\wedge\widehat{f_{i_m,j_m}}\wedge\ldots\otimes B_0 \label{epsilon} \\
&-\sum_{\ell,m=1}^{c-b+1}(-1)^{\ell+m}\epsilon_{m,l}b_{i_\ell,j_\ell}b_{i_m,j_m-1}f_{i_0,j+1}\wedge
\ldots\wedge\widehat{f_{i_\ell,j_\ell}}\wedge\widehat{f_{i_m,j_m}}\wedge\ldots\otimes B_0 \nonumber
\end{align}
Now the term (\ref{epsilon}) vanishes as swapping $m$ and $\ell$ gives terms that cancel. In fact in the second case we get the exact same expression but without the term (\ref{epsilon}).
If we again start in the lower left corner and then go to the right we get 
$$f_{i_0,j+1}\wedge f_{i_1,j_1}\wedge\ldots\wedge f_{i_{c-b+1},j_{c-b+1}}\otimes B_0,$$
and when we then go to the upper right and put $j_0=j+1$ we get
\begin{align}
&\sum_{0\leq a_1<a_2\leq c-b+1}(-1)^{a_1+a_2}(b_{i_{a_1},j_{a_1}}b_{i_{a_2},j_{a_2}-1}-b_{i_{a_1},j_{a_1}-1}b_{i_{a_2},j_{a_2}})
\nonumber \\
\cdot &f_{i_0,j_0}\wedge\ldots\wedge\widehat{f_{i_{a_1},j_{a_1}}}\wedge\ldots\wedge\widehat{f_{i_{a_2},j_{a_2}}}\wedge\ldots\wedge f_{i_{c-b+1},j_{c-b+1}}\otimes B_0.\nonumber
\end{align}
One checks that this is equal to the previous expression.
\section{Using the resolution}
From the resolution of $R_{a,b}$ one can in theory compute the entire graded Betti table and the bidegree table of each entry. One does this by tensoring   the resolution with the field $k$ over which we work and then taking homology. If we assign a bidegree decomposition to each free module in the resolution such that the maps preserve it, then we also get the bidegree table of each entry. Here we are interested in the $K_{p,1}$ which correspond to the degree $p+1$ part of the homology at place $p$.
\begin{equation}\label{mapoplus1}
\bigoplus_{\max(0,p+1-n)\leq j\leq\min(p+1,m)} Q_{p-j+1,j}\otimes k\longrightarrow
\end{equation}
\begin{equation}\label{mapoplus2}
\bigoplus_{\max(0,p-n)\leq j\leq\min(p,m)} Q_{p-j,j}\otimes k\longrightarrow\bigoplus_{\max(0,p-1-n)\leq j\leq\min(p-1,m)} Q_{p-j-1,j}\otimes k
\end{equation}
Where $m=(a+1)b-1$ and $n=a-1$. Each $Q_{p-j,j}$ with $p-j>0$ is a direct sum of copies of
\begin{align}
&\bigwedge^jF\otimes V_{(p-j+1)b-j}(-p-1)\hspace{43 pt}\text{if }j\leq(p-j+1)b,\nonumber \\
&\bigwedge^{j+1}F\otimes V_{j-(p-j+1)b-1}(-p-2)\hspace{30 pt}\text{if }j>(p-j+1)b.\nonumber
\end{align}
Now only those $Q_{p-j,j}$ with a degree shift $(-p-1)$ can contribute to $K_{p,1}$, as $K_{p,1}$ is the degree $p+1$ part of the homology of (\ref{mapoplus1}), (\ref{mapoplus2}). Moreover, the only way for a degree-preserving map from $Q_{p-j,j}$ to $Q_{p-1-j',j'}$ to not become zero upon tensoring with $k$ is if both have a degree shift $(-p-1)$, because otherwise basis elements of the domain get mapped to linear combinations of basis elements of the codomain where the coefficients have positive degree, and hence die when tensoring with $k$, killing the map. If $p-1-j'=0$ then the degree shift of $Q_{p-1-j',j'}$ is $(-p)$, while that of  $Q_{p-j,j}$ is $(-p-1)$ or $(-p-2)$, so then the map certainly becomes zero upon tensoring with $k$ by the same argument.\newline Therefore we only have to take into account morphisms from $Q_{p-j,j}$ to $Q_{p-1-j',j'}$ where 
\begin{equation}\label{nonzeromap}
j\leq(p-j+1)b\text{, }j'>(p-j')b\text{ and }p-1-j'>0.
\end{equation}
Combining the first inequality with the fact that $j\geq p-n$ and $p-j\leq n=a-1$ gives $p-n\leq ab$. So if $p>ab+a-1$ then the maps (\ref{mapoplus1}) and (\ref{mapoplus2}) are both zero and hence $K_{p,1}$ is just the degree $p+1$ part of the direct sum of the $Q_{p-j,j}\otimes k$. This is only non-zero for $Q_{0,p}$, because it has degree shift $(-p-1)$ and all the other $Q_{p-j,j}$ have degree shift $(-p-2)$. So we have
\begin{proposition}\label{usingres1}
For $p>ab+a-1$ we have $K_{p,1}\cong Q_{0,p}\otimes k=\bigwedge^{p+1}(F\otimes k)\otimes V_{p-1}$, which has dimension $p\binom{(a+1)b}{p+1}$.
\end{proposition}
This proves part $(i)$ of theorem \ref{firstrow}. We now study the case $p=ab+a-1$.\newline
We have to take the kernel of the map (\ref{mapoplus2}) divided by the image of the map (\ref{mapoplus1}). The latter is zero by the argument above, so we just have to take the kernel of the map (\ref{mapoplus2}). Let's study this map.\newline
For $p-j<n$ we still cannot have the first inequality in (\ref{nonzeromap}), so we only have to look at the maps from $Q_{n,p-n}$ to $Q_{p-j',j'}$. The kernel of the map (\ref{mapoplus2}) is the direct sum of $Q_{0,p}\otimes k$ with the kernel of the map
$$Q_{n,p-n}\otimes k\longrightarrow\bigoplus_{\max(0,p-1-n)\leq j\leq\min(p-1,m)} Q_{p-j-1,j}\otimes k$$
In stead of computing this latter kernel we will just compute the kernel of the map from $Q_{n,p-n}$ to $Q_{n-1,p-n}$ which contains it. This will then give an upper bound on $\kappa_{p,1}$ for $p=ab+a-1$. By lemma \ref{first_chain_map} this map is the first chain map, which we computed in the previous section. It comes from the first map in the relative resolution and the horizontal maps $\alpha_{j_0,p-n}$ from (\ref{horizontalmap}):
\footnotesize
\begin{align}
&\bigwedge^{ab}F\otimes\bigwedge^aG_{a}\otimes G_{a-1}(-p-1)\longrightarrow\bigwedge^{ab+1}F\otimes V_{b-1}\otimes\bigwedge^{a-1}G_{a}\otimes G_{a-2}(-p-1):\nonumber \\
&f_{i_1,j_1}\wedge\ldots\wedge f_{i_{ab},j_{ab}}\otimes g_1\wedge\ldots\wedge g_a\otimes g_h\longmapsto\nonumber \\
&\sum_{q=1}^a(-1)^{q+b-1}\sum_{\ell=0}^{b-1}f_{q-1,\ell+1}\wedge f_{i_1,j_1}\wedge\ldots\wedge f_{i_{ab},j_{ab}}\otimes B_{\ell}\otimes g_1\wedge\ldots\wedge\widehat{g_q}\wedge\ldots\wedge g_a\otimes g_{h-1}\nonumber \\
-&\sum_{q=1}^a(-1)^{q+b-1}\sum_{\ell=0}^{b-1}f_{q,\ell+1}\wedge f_{i_1,j_1}\wedge\ldots\wedge f_{i_{ab},j_{ab}}\otimes B_{\ell}\otimes g_1\wedge\ldots\wedge\widehat{g_q}\wedge\ldots\wedge g_a\otimes g_h.\nonumber
\end{align}
%(-1)^{b-1}\sum_{\ell=0}^{b-1}f_{j+\ell+1,j_0}\wedge f_{i_1,j_1}\wedge\ldots\wedge f_{i_{ab},j_{ab}}\otimes B_{\ell}
\normalsize
If $g_{h-1}$ or $g_h$ is out of range then that term is deleted. We have to tensor this map with $k$, so we interpret the $B_\ell$, the $f_{i,j}$ and the $g_h$ as vector space basis elements rather than free module basis elements. But we can change the map without changing the dimension of the kernel. One immediate simplification is to remove $\bigwedge^aG_{a}$ as it is one-dimensional:
\begin{align}
&\bigwedge^{ab}F\otimes G_{a-1}\longrightarrow\bigwedge^{ab+1}F\otimes V_{b-1}\otimes G_{a}\otimes G_{a-2}:\nonumber \\
&f_{i_1,j_1}\wedge\ldots\wedge f_{i_{ab},j_{ab}}\otimes g_h\longmapsto\nonumber \\
&\sum_{q=1}^a(-1)^{q+b-1}\sum_{\ell=0}^{b-1}f_{q-1,\ell+1}\wedge f_{i_1,j_1}\wedge\ldots\wedge f_{i_{ab},j_{ab}}\otimes B_{\ell}\otimes g_q\otimes g_{h-1}\nonumber \\
-&\sum_{q=1}^a(-1)^{q+b-1}\sum_{\ell=0}^{b-1}f_{q,\ell+1}\wedge f_{i_1,j_1}\wedge\ldots\wedge f_{i_{ab},j_{ab}}\otimes B_{\ell}\otimes g_q\otimes g_h.\nonumber
\end{align}
Another simplification is to remove the factor $(-1)^{q+b-1}$ as the $b$ is a constant and the $(-1)^q$ can be removed by mapping each basis element $g_q\in G_a$ to $(-1)^qg_q$, which is obviously a linear automorphism of $G_a$.
\begin{align}
&\bigwedge^{ab}F\otimes G_{a-1}\longrightarrow\bigwedge^{ab+1}F\otimes V_{b-1}\otimes G_{a}\otimes G_{a-2}:\nonumber \\
&f_{i_1,j_1}\wedge\ldots\wedge f_{i_{ab},j_{ab}}\otimes g_h\longmapsto\nonumber \\
&\sum_{q=1}^a\sum_{\ell=0}^{b-1}f_{q-1,\ell+1}\wedge f_{i_1,j_1}\wedge\ldots\wedge f_{i_{ab},j_{ab}}\otimes B_{\ell}\otimes g_q\otimes g_{h-1}\nonumber \\
-&\sum_{q=1}^a\sum_{\ell=0}^{b-1}f_{q,\ell+1}\wedge f_{i_1,j_1}\wedge\ldots\wedge f_{i_{ab},j_{ab}}\otimes B_{\ell}\otimes g_q\otimes g_h.\nonumber
\end{align}
Recall that the terms where $g_h$ or $g_{h-1}$ is out of range are removed. We now do a switch in notation: For any subset $S$ of $\ZZ^2$ we denote by $V_S$ a vector space with basis $S$ and we denote the basis element corresponding to $P\in S$ by $v_P$. So $F\otimes k$ can be identified with $V_{\{0,\ldots,a\}\times\{0,\ldots,b-1\}}$ so that $f_{i,j}$ corresponds to $v_{(i,j-1)}$. We also identify $V_{b-1}\otimes G_{a}$ with $V_{\{0,\ldots,a-1\}\times\{0,\ldots,b-1\}}$
where $B_{\ell}\otimes g_q$ corresponds to $v_{(q-1,\ell)}$. We get
\begin{align}
&\bigwedge^{ab}V_{\{0,\ldots,b-1\}\times\{0,\ldots,a\}}\otimes G_{a-1}\nonumber \\
\longrightarrow&\bigwedge^{ab+1}V_{\{0,\ldots,a\}\times\{0,\ldots,b-1\}}\otimes V_{\{0,\ldots,b-1\}\times\{0,\ldots,a-1\}}\otimes G_{a-2}:\nonumber \\
&v_{P_1}\wedge\ldots\wedge v_{P_{ab}}\otimes g_h\nonumber \\
\longmapsto&\sum_{P\in \{0,\ldots,a-1\}\times\{0,\ldots,b-1\}}v_{P}\wedge v_{P_1}\wedge\ldots\wedge v_{P_{ab}}\otimes v_P\otimes g_{h-1}\nonumber \\
-&\sum_{P\in \{0,\ldots,a-1\}\times\{0,\ldots,b-1\}}v_{P+(1,0)}\wedge v_{P_1}\wedge\ldots\wedge v_{P_{ab}}\otimes v_P\otimes g_h\nonumber
\end{align}
We do one final modification: $\bigwedge^{ab}V_{\{0,\ldots,a\}\times\{0,\ldots,b-1\}}$ is isomorphic to\newline $\bigwedge^{p-ab}V_{\{0,\ldots,a\}\times\{0,\ldots,b-1\}}$ by mapping each wedge product of the $v_P$ with distinct points $P$ in some set $S$ to the wedge product of those points not occurring in $S$. We similarly replace $\bigwedge^{ab+1}V_{\{0,\ldots,a\}\times\{0,\ldots,b-1\}}$ with $\bigwedge^{b-1}V_{\{0,\ldots,a\}\times\{0,\ldots,b-1\}}$. Note that these identifications depend on an ordering of the points of $\{0,\ldots,a\}\times\{0,\ldots,b-1\}$. Up to sign the map becomes:
\begin{align}
&\bigwedge^bV_{\{0,\ldots,a\}\times\{0,\ldots,b-1\}}\otimes G_{a-1}\nonumber \\
\longrightarrow&\bigwedge^{b-1}V_{\{0,\ldots,a\}\times\{0,\ldots,b-1\}}\otimes V_{\{0,\ldots,a-1\}\times\{0,\ldots,b-1\}}\otimes G_{a-2}:\nonumber \\
&v_{P_1}\wedge\ldots\wedge v_{P_b}\otimes g_h\nonumber \\
\longmapsto&\sum_{l=1}^b(-1)^lv_{P_1}\wedge\ldots\wedge\widehat{v_{P_l}}\wedge\ldots\wedge v_{P_b}\otimes v_{P_l}\otimes g_{h-1}\nonumber \\
-&\sum_{l=1}^b(-1)^lv_{P_1}\wedge\ldots\wedge\widehat{v_{P_l}}\wedge\ldots\wedge v_{P_b}\otimes v_{P_l-(1,0)}\otimes g_h\nonumber
\end{align}
where any terms that contain a $P_l$ or $P_{l-1}$ that is not in $\{0,\ldots,a-1\}\times\{0,\ldots,b-1\}$ are deleted. As before terms where $g_{h-1}$ or $g_h$ is not well defined are also deleted.
\begin{lemma}
A basis of the kernel of this map is given by the following expressions
\begin{equation}\label{kernelexpression}
\sum_{0\leq i_0,\ldots,i_{b-1}\leq a}v_{(i_0,0)}\wedge\ldots\wedge v_{(i_{b-1},b-1)}\otimes g_{h-i_0-\ldots-i_{b-1}}
\end{equation}
where $h\in\ZZ$ such that there is at least one term. The sum is over all choices of $0\leq i_0,\ldots,i_{b-1}\leq a$ such that $1\leq h-i_0-\ldots-i_{b-1}\leq a-1$.
\end{lemma}
\begin{proof}
Let us call the map $f$. We assume $a\geq 3$. We leave it as an exercise to the reader to show that these expressions are in $\ker f$. Let $x\in\ker f$, we will show it is of this form.\newline \underline{claim:} For any term $v_{P_1}\wedge\ldots\wedge v_{P_b}\otimes g_h$ occurring in $x$ with coefficient $\lambda$ the terms $v_{P_1}\wedge\ldots\wedge v_{P_i\pm(1,0)}\wedge\ldots\wedge v_{P_b}\otimes g_{h\mp 1}$ also occur in $x$ with coefficient $\lambda$, at least those that are in $\bigwedge^bV_{\{0,\ldots,a\}\times\{0,\ldots,b-1\}}\otimes G_{a-1}$. We now prove this claim. So let $1\leq i\leq b$. Applying $f$ to the term we get
\begin{align}
&\sum_{i=1}^b(-1)^iv_{P_1}\wedge\ldots\wedge\widehat{v_{P_i}}\wedge\ldots\wedge v_{P_{b}}\otimes v_{P_i}\otimes g_{h-1}\nonumber \\
-&\sum_{i=1}^b(-1)^iv_{P_1}\wedge\ldots\wedge\widehat{v_{P_i}}\wedge\ldots\wedge v_{P_{b}}\otimes v_{P_i-(1,0)}\otimes g_h.\nonumber
\end{align}
Suppose $P_i+(1,0)\in\{0,\ldots,a\}\times\{0,\ldots,b-1\}$ and $g_{h-1}\in G_a$, then the term $(-1)^iv_{P'_1}\wedge\ldots\wedge\widehat{v_{P'_i}}\wedge\ldots\wedge v_{P'_{b}}\otimes v_{P'_i+(1,0)}\otimes g_{h'}$ actually occurs in the above expression, so it has to cancel against something because $f(x)=0$. But the only thing it can cancel against is a term
$$-(-1)^iv_{P'_1}\wedge\ldots\wedge\widehat{v_{P'_i}}\wedge\ldots\wedge v_{P'_{b}}\otimes v_{P'_i-(1,0)}\otimes g_{h'}$$
occurring in $f(v_{P'_1}\wedge\ldots\wedge v_{P'_b}\otimes g_{h'})$ with $h'=h-1$, $P'_i=P_i+(1,0)$ and $P'_j=P_j$ for $j\neq i$. Therefore the term $$v_{P'_1}\wedge\ldots\wedge v_{P'_b}\otimes g_{h'}=v_{P_1}\wedge\ldots\wedge_{P_i+(1,0)}\wedge\ldots\wedge v_{P_b}\otimes g_{h-1}$$
must occur in $x$ with coefficient $\lambda$. By an similar argument the term $v_{P_1}\wedge\ldots\wedge_{P_i-(1,0)}\wedge\ldots\wedge v_{P_b}\otimes g_{h+1}$ must also occur in $x$ with coefficient $\lambda$ if $P_i-(1,0)\in\{0,\ldots,a\}\times\{0,\ldots,b-1\}$ and $g_{h+1}\in G_a$, proving the claim.\newline
We now show that any term of $x$ is of the form $v_{P_1}\wedge\ldots\wedge v_{P_b}\otimes g_h$ where all the $P_i$ have distinct second coordinate. Because if $P_1$ and $P_2$ have the same second coordinate, say $P_1=(d_1,c)$ and $P_2=(d_2,c)$ with $d_1<d_2$ then by the claim at least one of $v_{(d_1+1,c)}\wedge v_{(d_2,c)}\wedge\ldots\wedge v_{P_b}\otimes g_{h-1}$ and $v_{(d_1,c)}\wedge v_{(d_2-1,c)}\wedge\ldots\wedge v_{P_b}\otimes g_{h+1}$ occurs in $x$. In the first case one repeatedly applies the claim and all of the following terms must occur in $x$:
\begin{align}
&v_{(d_1+1,c)}\wedge v_{(d_2-1,c)}\wedge\ldots\wedge v_{P_b}\otimes g_{h}\nonumber \\
&v_{(d_1+2,c)}\wedge v_{(d_2-1,c)}\wedge\ldots\wedge v_{P_b}\otimes g_{h-1}\nonumber \\
&v_{(d_1+2,c)}\wedge v_{(d_2-2,c)}\wedge\ldots\wedge v_{P_b}\otimes g_{h}\nonumber
\end{align}
etcetera. After a finite number of steps the first and second wedge factor coincide, producing a contradiction. The second case is handled similarly.\newline We conclude that all of the $P_i$ in any term of $x$ have distinct second coordinate, so any term of $x$ may be written as
$$v_{(l_0,0)}\wedge\ldots\wedge v_{(l_{b-1},b-1)}\otimes g_h.$$
By applying the claim again we will prove that any two expressions of this form with the same value of $l_0+\ldots+l_{b-1}+h$ must have the same coefficient in $x$. If we can prove this we are almost done. So let 
$$v_{(l'_0,0)}\wedge\ldots\wedge v_{(l'_{b-1},b-1)}\otimes g_{h'},\hspace{20 pt}l_0+\ldots+l_{b-1}+h=l'_0+\ldots+l'_{b-1}+h'$$
be another such expression and suppose for instance that $h'>h$. Then we can repeatedly add one to $h$ while substracting 1 from some $l_i$ that is greater than $l'_i$, until $h$ is equal to $h'$. By the claim this process does not change the coefficient of this term in $x$. If the first expression is not yet equal to the second expression then we choose an $i$ such that $l_i<l'_i$, we add one to $l_i$ while substracting one from $h$, then we choose a $j$ such that $l_j>l'_j$ and we substract 1 from $l_j$ while adding 1 to $h$. We keep repeating this process until both expressions are the same. The conclusion is that all terms occurring in (\ref{kernelexpression}) have the same coefficient in $x$, so $x$ is a linear combination of expressions like (\ref{kernelexpression}). That these expressions are linearly independent follows from the fact that they have no terms in common.
\end{proof}
\begin{proposition}\label{usingres2}
If $3\leq a\leq b$ and $p=ab+a-1$ then
$$\kappa_{p,1}\leq p\binom{(a+1)b}{p+1}+p.$$
\end{proposition}
\begin{proof}
By the discussion after proposition \ref{usingres1} we know that $K_{p,1}$ is the direct sum of $Q_{0,p}\otimes k$ and something contained in the kernel of the map that we modified and whose kernel we computed in the previous lemma. Now $Q_{0,p}=\bigwedge^{p+1}F\otimes V_{p-1}(-p-1)$, so $Q_{0,p}\otimes k$ has dimension $p\binom{(a+1)b}{p+1}$. It remains to compute the number of expressions (\ref{kernelexpression}) in the previous lemma. an integer $h$ gives a valid expression if and only if there exist $0\leq j_0,\ldots,j_{b-1}\leq a-1$ such that $1\leq h-j_0-\ldots-j_{b-1}\leq a$, which happens if and only if $1\leq h\leq a(b+1)-1$. This means that the kernel of the map has dimension $a(b+1)-1=p$.
\end{proof}
Note that the conclusion is false when $a=2$. So how did we use that $a\geq 3$? We computed a horizontal chain map from the $a-1$-th module to the $a-2$-th module of the relative resolution. Now the zero-th module is of a different kind than all the other modules. When $a=2$ the chain map goes from the first module to the zero-th module, and in fact the chain map becomes zero when tensoring with $k$.
\section{Explicit Koszul cocycles}
At the end of this section we will finally prove theorem \ref{firstrow}.
Let $M$ be a graded module over a polynomial ring $R$ over a field $k$. There are two ways of computing $K_{p,q}(M):=(\textup{Tor}^p_R(M,k))_{p+q}$. By taking a free graded resolution of $M$, then tensoring with $k$ and taking homology, or by taking a free graded resolution of $k$, tensoring with $M$ and taking homology (see \cite[Theorem 1.b.4 p.\ 133]{greenkoszul}). Denote by $V$ the degree one part of $R$. The following diagram can be used to prove the equivalence of both definitions.
\begin{equation}\label{stairs}
\begin{tikzcd}
\ddots & \vdots\arrow{d}& \vdots\arrow{d}& \vdots\arrow{d}& \vdots\arrow{d} \\
\ldots\arrow{r}& \bigwedge^2V \otimes P_2\arrow{r}\arrow{d}& \bigwedge^2V\otimes P_1\arrow{r}\arrow{d}& \bigwedge^2V\otimes P_0\arrow{r}\arrow{d}& \bigwedge^2V\otimes M \arrow{d}\\
\ldots\arrow{r}& V\otimes P_2\arrow{r}\arrow{d}& V\otimes P_1\arrow{r}\arrow{d}& V\otimes P_0\arrow{r}\arrow{d}& V\otimes M\arrow{d} \\
\ldots\arrow{r}& P_2\arrow{r}\arrow{d}& P_1\arrow{r}\arrow{d}& P_0\arrow{r}\arrow{d}& M\arrow{d} \\
\ldots\arrow{r}& k\otimes P_2\arrow{r}& k\otimes P_1\arrow{r}& k\otimes P_0\arrow{r}& k\otimes M
\end{tikzcd}
\end{equation}
This diagram is obtained by tensoring the resolution $\ldots\rightarrow P_2\rightarrow P_1\rightarrow P_0\rightarrow M$ with the resolution $\ldots\rightarrow \bigwedge^2V\otimes R\rightarrow V\otimes R\rightarrow R\rightarrow k$. Actually we deleted some degree shifts from our notation. The idea is that the homology of the bottom line is isomorphic to that of the rightmost column, as one can prove with some diagram chasing.\newline
We will use this diagram to construct explicit Koszul cocycles of the $K_{p,1}$ for $p=ab+a-1$ and we will use these to prove that the inequality of proposition \ref{usingres2} is an equality.
The idea is to embed the $K_{p,1}$ of the rational normal scrolls $\PP^a\times\PP^1$ and $\PP^1\times\PP^b$ into the $K_{p,1}$ of $\PP^1\times\PP^1$. We denote their homogeneous coordinate rings $R_{\bar{b}}$, $R_{\bar{a}}$ and $R_{a,b}$ respectively. So we are projectively embedding our $\PP^1\times\PP^1\rightarrow\PP^{N-1}$ into two rational normal scrolls of dimension $a+1$ and $b+1$ respectively. By \cite[Corollary 1.26, $(ii)$ p.\ 8]{nagelaprodu} the induced morphisms
$$K_{p,1}(R_{\bar{a}})\longrightarrow K_{p,1}(R_{a,b}),\hspace{20 pt}K_{p,1}(R_{\bar{b}})\longrightarrow K_{p,1}(R_{a,b})$$
are injective. This is because the first two horizontal maps in the following diagram are isomorphisms
\begin{center}
\begin{tikzcd}
\bigwedge^{p+1}V\otimes(R_{\bar{b}})_0 \arrow{r}\arrow{d}& \bigwedge^pV\otimes(R_{\bar{b}})_1 \arrow{r}\arrow{d}& \bigwedge^{p-1}V\otimes(R_{\bar{b}})_2 \arrow{d} \\
\bigwedge^{p+1}V\otimes (R_{a,b})_0 \arrow{r}& \bigwedge^pV\otimes (R_{a,b})_1 \arrow{r}& \bigwedge^{p-1}V\otimes (R_{a,b})_2
\end{tikzcd}
\end{center}
hence the induced map on homology is injective.
We now compute an explicit basis of $R_{\bar{b}}$ using the Eagon-Northcott resolution. the same can then be done with $R_{\bar{a}}$. Let $f_{i_1,j_1}\wedge\ldots\wedge f_{i_{p+1},j_{p+1}}\otimes B_j$ be an element of $P_p\otimes k=\bigwedge^{p+1} F\otimes V_{p-1}\otimes k$. We now navigate the diagram (\ref{stairs}) until we arrive at an element of $\bigwedge^{p-1}V\otimes (R)_2=\bigwedge^{p-1}V\otimes S^2V$.
So we will explicitly compute the isomorphism
$$\bigwedge^{p+1} F\otimes V_{p-1}\otimes k\cong\ker\Big(\bigwedge^{p-1}V\otimes S^2V\rightarrow\bigwedge^{p-1}V\otimes(R_{\bar{a}})_2\Big),$$
which is isomorphic to the Koszul cohomology.
The journey we will make can be visualized in the following diagram which is obtained from the diagam (\ref{stairs}) by taking the degree $p+1$ part of every object, (taking into account the degree shifts that aren't shown):
\small
\begin{center}
\begin{tikzcd}
 \bigwedge^3 V\otimes \bigwedge^{p-2}F\otimes V_{p-4}\arrow{r}\arrow[hook]{dr}& \ldots\\
 \bigwedge^2V\otimes \bigwedge^{p-1}F\otimes V_{p-3} \arrow{r}\arrow[hook]{dr}& \bigwedge^2V\otimes V\otimes \bigwedge^{p-2}F\otimes V_{p-4} \\
 V\otimes \bigwedge^{p}F\otimes V_{p-2} \arrow{r}\arrow[hook]{dr}& V\otimes V\otimes \bigwedge^{p-1}F\otimes V_{p-3}  \\
\bigwedge^{p+1}F\otimes V_{p-1}\arrow[hook]{dr}{=}\arrow{r} & V\otimes \bigwedge^{p}F\otimes V_{p-2} \\
& k\otimes \bigwedge^{p+1}F\otimes V_{p-1} 
\end{tikzcd}
\end{center}
\normalsize
We start from the bottom and go to the top. When we go to $V\otimes \bigwedge^{p}F\otimes V_{p-2}$ we get
\begin{align}
&\sum_l(-1)^lb_{i_l,j_l-1}\otimes f_{i_1,j_1}\wedge\ldots\wedge\widehat{f_{i_l,j_l}}\wedge\ldots\wedge f_{i_{p+1},j_{p+1}}\otimes B_{j-1}\nonumber \\
-&\sum_l(-1)^lb_{i_l,j_l}\otimes f_{i_1,j_1}\wedge\ldots\wedge\widehat{f_{i_l,j_l}}\wedge\ldots\wedge f_{i_{p+1},j_{p+1}}\otimes B_{j}.\nonumber
\end{align}
We now give the general formula for the element of $\bigwedge^mV\otimes\bigwedge^{p-m+1}F\otimes V_{p-m-1}$:
\begin{align}
&\sum_{e_1,\ldots,e_m\in\{0,1\}}\sum_{l_1>\ldots>l_m\geq 1}^{p+1}(-1)^{e_1+\ldots+e_m+l_1+\ldots+l_m}b_{i_{l_1},j_{l_1}-e_1}\wedge\ldots\wedge b_{i_{l_m},j_{l_m}-e_m}\nonumber \\
&\otimes f_{i_1,j_1}\wedge\ldots\wedge\widehat{f_{i_{l_m},j_{l_m}}}\wedge\ldots\wedge\widehat{f_{i_{l_1},j_{l_1}}}\wedge\ldots\wedge f_{i_{p+1},j_{p+1}}\otimes B_{j-e_1-\ldots-e_m}.\nonumber
\end{align}
Here we only take those $e_1,\ldots,e_m\in\{0,1\}$ such that $0\leq j-e_1-\ldots-e_m\leq p-m-1$ so that $B_{j-e_1-\ldots-e_m}$ is well defined. When $m=p-1$ and we go from $\bigwedge^{p-1}V\otimes\bigwedge^2F$ to $\bigwedge^{p-1}V\otimes S^2V$ we get
\begin{align}
&\sum_{e_1,\ldots,e_{p-1}\in\{0,1\}}\sum_{l_1>\ldots>l_{p-1}\geq 1}^{p+1}(-1)^{j+l_1+\ldots+l_{p-1}}b_{i_{l_1},j_{l_1}-e_1}\wedge\ldots\wedge b_{i_{l_{p-1}},j_{l_{p-1}}-e_{p-1}}\nonumber \\
&\otimes b_{i_{l'_1},j_{l'_1}-1}b_{i_{l'_2},j_{l'_2}}-b_{i_{l'_1},j_{l'_1}}b_{i_{l'_2},j_{l'_2}-1},\nonumber
\end{align}
where the sum is over all $e_1,\ldots e_{p-1}\in\{0,1\}$ summing up to $j$, and $l'_1<l'_2$ are the elements of $\{1,\ldots,p+1\}$ missing from the list $l_1,\ldots,l_{p-1}$. With diagram chasing one proves that these expressions are linearly independent, because they come from linearly independent elements of $\bigwedge^{p+1}F\otimes V_{p-1}$. So we will apply this to two rational normal scrolls with homogeneous coordinate rings $R_{\bar{b}}$ and $R_{\bar{a}}$ respectively. We now change notation: for any finite set $S$ of points we write $V_S$ for a vector space with a basis whose elements we denote by $v_P$, for all $P\in S$. For any set $\{P_1,\ldots,P_{p+1}\}$ of distinct points in $\{0,\ldots,a\}\times\{0,\ldots,b-1\}$ and any $0\leq j\leq p-1$ we have an element
\begin{align}
&\underset{e_1+\ldots+e_{p-1}=(0,j)}{\sum_{e_1,\ldots,e_{p-1}\in\{(0,0),(0,1)\}}}\sum_{l_1>\ldots>l_{p-1}\geq 1}^{p+1}(-1)^{l_1+\ldots+l_{p-1}}v_{P_{l_1}+e_1}\wedge\ldots\wedge v_{P_{l_{p-1}}+e_{p-1}}\nonumber \\
&\otimes v_{P_{l'_1}+(0,1)}v_{P_{l'_2}}-v_{P_{l'_1}}v_{P_{l'_2}+(0,1)}\label{scrollexpression1}
\end{align}
of $\bigwedge^{p-1}V\otimes S^2V$, and for any $0\leq j\leq p-1$ and some fixed ordering $P_1,\ldots,P_{p+1}$ of the elements of $\{0,\ldots,a-1\}\times\{0,\ldots,b\}$ we also have an element
\begin{align}
&\underset{e_1+\ldots+e_{p-1}=(j,0)}{\sum_{e_1,\ldots,e_{p-1}\in\{(0,0),(1,0)\}}}\sum_{l_1>\ldots>l_{p-1}\geq 1}^{p+1}(-1)^{l_1+\ldots+l_{p-1}}v_{P_{l_1}+e_1}\wedge\ldots\wedge v_{P_{l_{p-1}}+e_{p-1}}\nonumber \\
&\otimes v_{P_{l'_1}+(1,0)}v_{P_{l'_2}}-v_{P_{l'_1}}v_{P_{l'_2}+(1,0)}\label{scrollexpression2}
\end{align}
of $\bigwedge^{p-1}V\otimes S^2V$. Note that $\{0,\ldots,a-1\}\times\{0,\ldots,b\}$ has $p-1$ elements, so we just have to fix an ordering of the points.
\begin{lemma}\label{independent}
The elements of $\bigwedge^{p-1}V\otimes S^2V$ in (\ref{scrollexpression1}) and (\ref{scrollexpression2}) are linearly independent.
\end{lemma}
\begin{proof}
We already remarked that the expressions (\ref{scrollexpression1}) are linearly independent. Now $\bigwedge^{p-1}V\otimes S^2V$ can be decomposed as a direct sum $\bigoplus_{u\in\ZZ^2}(\bigwedge^{p-1}V\otimes S^2V)_u$ of bidegree components. So $(\bigwedge^{p-1}V\otimes S^2V)_u$ is generated by the $P_1\wedge\ldots\wedge P_{p-1}\otimes Q_1Q_2$ that satisfy $P_1+\ldots+P_{p-1}+Q_1+Q_2=u$. So all we have to prove is that for every $u$ all expressions like (\ref{scrollexpression1}) and (\ref{scrollexpression2}) that are in $(\bigwedge^{p-1}V\otimes S^2V)_u$ are linearly independent. But for those in (\ref{scrollexpression2}) the bidegree is always distinct: $$u=(0,j+1)+\sum_{P\in\{0,\ldots,a-1\}\times\{0,\ldots,b\}}P=(a(a-1)(b+1)/2+j+1,ab(b+1)/2).$$
Therefore it is enough to prove that no expression as in (\ref{scrollexpression2}) is linearly generated by the expressions in (\ref{scrollexpression1}).\newline
If we can prove the following two claims we are done:\newline
\underline{Claim 1:} Every expression in (\ref{scrollexpression2}) has a term which contains all points of some column of $\{0,\ldots,a\}\times\{0,\ldots,b\}$ in its wedge part.\newline
\underline{Claim 2} No expression as in (\ref{scrollexpression1}) has a term which contains all points of some column of $\{0,\ldots,a\}\times\{0,\ldots,b\}$ in its wedge part.\newline
We now prove claim 1 by constructing such a term for every $0\leq j\leq p-1$. We order the points of $\{0,\ldots,a-1\}\times\{0,\ldots,b\}$ by putting $P_{y+(b+1)x+1}=(x,y)$ for $0\leq y\leq b$ and $0\leq x<a$. We set $e_1=\ldots=e_{p-j-1}=(0,0)$ and $e_{p-j}=\ldots=e_{p-1}=(1,0)$ and for $l'_1$ and $l'_2$ we take $p-j$ and $p-j+1$. This determines a term of (\ref{scrollexpression2}) and the set of points occurring in the wedge part is
\begin{equation}\label{rectangleset}
\{(x,y)\in\{0,\ldots,a\}\times\{0,\ldots,b\}|y+(b+1)x<p-j-1\text{ or }y+(b+1)x>p+b-j+1\}.
\end{equation}
If $j>b$ then the column consisting of $(a,0),\ldots,(a,b)$ is contained in this set. If $j<p-b-1$ then the column consisting of $(0,0),\ldots,(0,b)$ is contained in this set. At least one of these inequalities is satisfied for every $j$ since $b+1<p-b-1$. Here we use that $3\leq a\leq b$. To prove claim 1 it remains to show that this term does not cancel against any other term of the expression. The set of points in the wedge part of any non-zero term in the expression (\ref{scrollexpression2}) can be obtained by starting from the set $\{0,\ldots,a-1\}\times\{0,\ldots,b\}$, removing two points and moving some of the remaining points one to the right (adding (1,0) to them) without any collisions. For any $y\in\{0,\ldots,b\}$ the set of points $(x,y)$ in the set (\ref{rectangleset}) is of the form $\{(0,y),\ldots,(a,y)\}\backslash\{(x_1,y),\ldots,(x_2,y)\}$ with $0\leq x_2-x_1\leq 2$, which can only come from $\{(0,y),\ldots,(a-1,y)\}\backslash\{(x_1,y),\ldots,(x_2-1,y)\}$ (or $\{(0,y),\ldots,(a-1,y)\}$ if $x_1=x_2$) by moving $(x_2,y),\ldots,(a-1,y)$ to the right. So the term (\ref{rectangleset}) can only be obtained in one way, and hence cannot cancel against another term.
For instance in the following picture with $a=b=3$
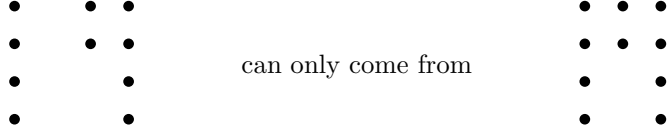
\begin{center}
\begin{tikzpicture}
\fill[black] (.5,.5) circle (2 pt);
\fill[black] (.5,0) circle (2 pt);
\fill[black] (2,.5) circle (2 pt);
\fill[black] (2,0) circle (2 pt);
\fill[black] (1.5,1) circle (2 pt);
\fill[black] (2,1) circle (2 pt);
\fill[black] (.5,1.5) circle (2 pt);
\fill[black] (.5,1) circle (2 pt);
\fill[black] (1.5,1.5) circle (2 pt);
\fill[black] (2,1.5) circle (2 pt);
\node at (5,.7) {can only come from};
\fill[black] (8,0) circle (2 pt);
\fill[black] (9,.5) circle (2 pt);
\fill[black] (9,0) circle (2 pt);
\fill[black] (8,.5) circle (2 pt);
\fill[black] (8,1.5) circle (2 pt);
\fill[black] (8,1) circle (2 pt);
\fill[black] (8.5,1) circle (2 pt);
\fill[black] (9,1) circle (2 pt);
\fill[black] (8.5,1.5) circle (2 pt);
\fill[black] (9,1.5) circle (2 pt);
\end{tikzpicture}
\end{center}
This concludes the proof of the first claim. We now prove the second claim. For any term in the expression (\ref{scrollexpression1}) the set of points occurring in the wedge part can be obtained by starting from the set $\{P_1,\ldots,P_{p+1}\}\subset\{0,\ldots,a\}\times\{0,\ldots,b-1\}$, removing two points and then moving some points up (adding (0,1)) without any collisions. Since every column initially contains at most $b$ points you can never get a full column with $b+1$ points by moving points up. This concludes the proof of the second claim and hence the lemma.
\end{proof}
Finally we prove theorem \ref{firstrow}.
\begin{proof}[Proof of theorem \ref{firstrow}]
We already proved the first statement in proposition \ref{usingres1}. We now prove the second statement. By lemma \ref{usingres2} we have that $\kappa_{p,1}\leq p\binom{(a+1)b}{p+1}+p$, so we only have to prove the other inequality. We claim that $K_{p,1}$ is isomorphic to the kernel of $\bigwedge^{p-1}V\otimes I\rightarrow\bigwedge^{p-2}V\otimes S^3V$ where $I\subseteq S^2V$ consists of the quadrics generating the ideal of $\PP^1\times\PP^1$ in $\PP^{N-1}$. This follows from some diagram chasing in the following diagram:
\begin{center}
\begin{tikzcd}
0\arrow{r} & \bigwedge^{p+1}V \arrow{r}\arrow{d} & \bigwedge^{p+1}V\arrow{d} \\
0\arrow{r} & \bigwedge^{p}V\otimes V \arrow{r}\arrow{d} & \bigwedge^{p}V\otimes V\arrow{d} \\
\bigwedge^{p-1}V\otimes I \arrow{r} & \bigwedge^{p-1}V\otimes S^2V \arrow{r}\arrow{d} & \bigwedge^{p-1}V\otimes R_2\arrow{d} \\
 & \bigwedge^{p-2}V\otimes S^3V \arrow{r} & \bigwedge^{p-2}V\otimes R_3 
\end{tikzcd}
\end{center}
where $R$ is the homogeneous coordinate ring of $\PP^1\times\PP^1$. We leave the details to the reader. Now the expressions in lemma \ref{independent} are elements of this kernel, and by the lemma they are linearly independent. The number of such expressions is exactly $p\binom{(a+1)b}{p+1}+p$, which is therefore a lower bound on $\kappa_{p,1}$. This concludes the proof.
\end{proof}
\section{Bidegree tables}
The $k$-algebras we have been working with aren't just graded, each graded piece also has a bidegree decomposition. For instance the degree one part of $R_{a,b}$ is a vector space with basis elements $v_{(i,j)}$ where $0\leq i\leq a$ and $0\leq j\leq b$. So it decomposes as a direct sum of one-dimensional subspaces which are eigenspaces of the torus action on $\PP^1\times\PP^1$. Each $v_{(i,j)}$ is said to have bidegree $(i,j)$. So the polynomial ring in $(a+1)(b+1)$ variables and all the modules we have considered aren't just graded but their graded pieces all have a bidegree decomposition. So the syzygy spaces $K_{p,q}$ also have such a decomposition and this can be captured in the bidegree table. In the introduction we already showed the bidgree table of $K_{11,1}$ for $a=b=3$.
\begin{proposition}\label{cross}
In the case $3\leq a=b$ the bidegree table of $\kappa_{p,1}$ with $p=a^2+a-1$ is 2 at place $(a^2(a+1/2),a^2(a+1)/2)$, it takes the value 1 at points where one coordinate is $a^2(a+1)/2$ and the other is different from $a^2(a+1)/2$ but with the difference being smaller than $a(a+1)/2$. Everywhere else it is zero.
\end{proposition}
In other words the the cross pattern shown in the beginnning of the article occurs whenever $3\leq=b$.
\begin{proof}
The expressions (\ref{scrollexpression1}) determined by the number $j$ have bidegree $(a(a^2-1)/2+j+1,a^2(a+1)/2)$ and the expressions (\ref{scrollexpression2}) determined by the number $j$ have bidegree $(a^2(a+1)/2,a(a^2-1)/2+j+1)$. Together these form a basis of $K_{p,1}$ by the proof of theorem \ref{firstrow}.
\end{proof}
We conclude with two examples. Both tables are rotated 90 degrees. The following is $a=3$, $b=4$, $p=15$:
\begin{center}
\begin{tabular}{ccccccccccccccccc}
0&0&0&0&0&0&0&0&0&0&0&0&0&0&0&0&0\\
0&1&1&1&1&1&1&1&1&1&1&1&1&1&1&1&0\\
0&0&0&0&0&0&0&0&0&0&0&0&0&0&0&0&0\\
\end{tabular}
\end{center}
We get such a line whenver $p=ab+b-1$ and $3\leq a<b$.
The following is more interesting, namely $a=3$, $b=4$, $p=14$:
\begin{center}
\begin{tabular}{ccccccccccccccccc}
0&0&0&0&0&0&0&0&1&0&0&0&0&0&0&0&0\\
0&0&0&0&0&0&0&0&1&0&0&0&0&0&0&0&0\\
0&0&0&0&0&0&0&0&1&0&0&0&0&0&0&0&0\\
0&0&0&0&0&0&0&0&1&0&0&0&0&0&0&0&0\\
0&0&0&0&0&0&0&0&1&0&0&0&0&0&0&0&0\\
1&2&3&4&4&4&4&4&5&4&4&4&4&4&3&2&1\\
1&2&3&4&4&4&4&4&5&4&4&4&4&4&3&2&1\\
1&2&3&4&4&4&4&4&5&4&4&4&4&4&3&2&1\\
1&2&3&4&4&4&4&4&5&4&4&4&4&4&3&2&1\\
0&0&0&0&0&0&0&0&1&0&0&0&0&0&0&0&0\\
0&0&0&0&0&0&0&0&1&0&0&0&0&0&0&0&0\\
0&0&0&0&0&0&0&0&1&0&0&0&0&0&0&0&0\\
0&0&0&0&0&0&0&0&1&0&0&0&0&0&0&0&0\\
0&0&0&0&0&0&0&0&1&0&0&0&0&0&0&0&0\\
\end{tabular}
\end{center}
Obviously this shape comes from the two sets of syzygies from the embeddings of $\PP^1\times\PP^1$ into $\PP^a\times \PP^1$ and $\PP^1\times\PP^b$.\\

Most of the time bidegree tables of toric surfaces have a convex shape. This gives a class of counterexamples, namely for $p=ab+a-1$, whenever $3\leq a\leq b$. Of course there are other counterexamples as well. But here there is a very natural explanation for the odd shapes, namely the syzygies are coming from two different embeddings of $\PP^1\times\PP^1$ into rational normal scrolls.

\end{document}